\documentclass{birkjour}
\usepackage{amsmath, amssymb, amscd, amsthm,amsfonts}

\theoremstyle{plain}
\newtheorem{thm}{Theorem}[section]
\newtheorem{cor}[thm]{Corollary}
\newtheorem{lem}[thm]{Lemma}
\newtheorem{prop}[thm]{Proposition}
\newtheorem{conj}[thm]{Conjecture}
\theoremstyle{definition}
\newtheorem{defn}[thm]{Definition}
\theoremstyle{remark}
\newtheorem{rem}[thm]{Remark}

\numberwithin{equation}{section}

\begin{document}

%
%
%
%
%
%
%
%
%

\def \al {{\alpha}}
\def \T {{\mathbb T}}
\def \D {{\mathbb D}}
\def \C {{\mathbb C}}

\title[Invariance in $H^\al$ spaces]{Multiplication by finite Blaschke factors on a general class of Hardy spaces}

\author[A. Singh]{Apoorva Singh}

\address{Department of Mathematics\\
School of Natural Sciences\\
Shiv Nadar University\\
Gautam Budh Nagar - 201314.\\
Uttar Pradesh, India}

\email{as721@snu.edu.in}

\author[N. Sahni]{Niteesh Sahni}
\address{Department of Mathematics\br
School of Natural Sciences\br
Shiv Nadar University\br
Gautam Budh Nagar - 201314.\br
Uttar Pradesh, India}
\email{niteesh.sahni@snu.edu.in}


\begin{abstract}
A broader class of Hardy spaces and Lebesgue spaces have been introduced recently on the unit circle by considering continuous $\|.\|_1$-dominating normalized gauge norms instead of the classical norms on measurable functions and a Beurling type result has been proved for the operator of multiplication by the coordinate function. In this paper, we generalize the above Beurling type result to the context of multiplication by a finite Blaschke factor $B(z)$ and also derive the common invariant subspaces of  $B^2(z)$ and $B^3(z)$. These results lead to a factorization result for all functions in the Hardy space equipped with a continuous rotationally symmetric norm.
\end{abstract}

\maketitle
\section{Introduction}\label{sec1}
The study of invariant subspaces of the shift operator has been an area of extensive research over the past decades. After the landmark success of obtaining the invariant subspaces of $H^2$ under the multiplication with the monomial $z$ by Beurling \cite{Beurling1948}, many new directions of extensions have emerged. Some of the notable ones are the Helson-Lowdenslager theorem \cite{invariant_subspaces_Helson} which characterizes the simply invariant subspaces of $L^2$, Weiner’s theorem \cite{helson2017lectures} which characterizes the doubly invariant subspaces of $L^2$, the theorems by Lax \cite{lax1959translation} and Halmos \cite{Halmos1961Shifts} which study the invariance under the operator of multiplication by $z^n$, and the characterization of contractively contained sub Hilbert spaces of $H^2$ by de Branges \cite{de2015square}. Furthermore, in the recent past, the validity of the above results have been proved in the context of general Lebesgue $L^p$ and Hardy spaces $H^p$  $(0 \leq p \leq \infty)$ by \cite{Srinivasan1963simply}, \cite{Srinivasan1964doubly}, \cite{Lance1997multiplication}, \cite{Paulsen2001debranges}, and \cite{Agarwal1995debranges}. It is worth pointing out that in the works cited above, the authors have studied invariance under the multiplier algebra $H^\infty(z)$. For a finite Blaschke factor $B(z)$, invariance in $H^p$ and $L^p$ have been explored under the algebra $H^\infty (B)$ = $\{ f(B(z)) : f \in H^\infty\}$ in \cite{sahni2012lax}, \cite{Kumar2016invariance}, and \cite{SnehLata2010finite}.
\par
The common invariant subspaces have also been a focus of attention in the recent past. This encapsulates the discovery of invariant subspaces of $H^p$ under the algebras $H_1^\infty (z)$ = $\{ f \in H^\infty : f^{\prime}(0)=0 \}$, $H_1^\infty (B)$ = $\{ f(B(z)) : f \in H^\infty_1\}$, and $ \C + BH^\infty (z)$. Some notable works on the same are \cite{Davidson2009constrained}, \cite{sahni2012lax}, and \cite{raghupati2009nevalinna}.
\par
In \cite{chen2017general}, Chen defined the continuous $\|.\|_1$-dominating normalized gauge norm $\al$ on the space of all measurable functions on the unit circle. The closure of $L^\infty$ with respect to the norm $\al$ is called the Lebesgue space and is denoted by $L^\al$, and the closure of $H^\infty$ under $\al$ is called the Hardy space $H^\al$. It has been proved in \cite{chen2017general} that both $L^\al$ and $H^\al$ are Banach spaces. For all $1 \leq p <\infty$, the classical Lebesgue spaces $L^p$ and the Hardy spaces $H^p$ on the unit circle are special cases of $L^\al$ and $H^\al$ respectively. It must be noted that $H^\infty$ does not belong to these special classes. In fact $H^\infty \subsetneq H^\al$. The simply invariant and doubly invariant subspaces of multiplication by the coordinate function $e^{i\theta}$ on $L^\al$ have also been obtained \cite{chen2017general}. Beurling's theorem, as usual, follows as a special case of the $L^\al$ result. 
\par
The statement of the general Beurling's theorem in  \cite{chen2017general} is:
\begin{thm}\label{Theorem1.1}
Suppose $\al$ is a continuous $\|.\|_1$-dominating normalized gauge norm and $\mathcal{M}$ is a non-trivial closed subspace of $H^\al$ invariant under the multiplication by the coordination function $e^{i\theta}$. Then, there is an inner function $\phi$ such that $\mathcal{M} = \phi H^\al$. (A function $\phi \in H^\infty$ is called inner if $\lvert \phi \rvert =1$ a.e.)
\end{thm}
The key components of the proof for Theorem~{\ref{Theorem1.1}} includes the description of the dual space of $L^\al$ and a factorization theorem on $L^\infty$. 
\par
In this paper, we study the invariance in a general class of Hardy spaces equipped with continuous $\|.\|_1$-dominating normalized gauge norms introduced by Chen \cite{chen2017general} in the context of the algebras $H^\infty(B)$ and $H_1^\infty(B)$. Results are also obtained for a special class of continuous rotationally symmetric norms \cite{Chen2014Thesis}, \cite{Chen2014LebesgueAH}. Furthermore, we present a general inner-outer factorization of $H^\al$ functions in terms of $n$-inner and $n$-outer functions when $\al$ is a continuous rotationally symmetric norm. This result is a generalization of the inner-outer factorization of functions in $H^2$ obtained in \cite{singh1997multiplication}.
\par
The proofs presented in the current paper do not rely on any factorization results concerning $L^\infty$ functions. In that sense, our arguments are more elementary than those in \cite{chen2017general}. We only make use of the well known inner-outer factorization of $H^2$ functions. This allows the techniques to work in the context of a more general algebra $H^\infty(B)$. We further describe the common invariant subspaces of $H^\al$. We also obtain sharp descriptions of the invariant and common invariant subspaces in the rotationally symmetric case. This generalizes the main results of \cite{chen2017general}, \cite{Chen2014LebesgueAH} in addition to the main results in \cite{Lance1997multiplication} and \cite{sahni2012lax}.
\par
The rest of the paper has been organized as follows. Section~\ref{sec2} deals with the basic definitions and preliminary results of classical Hardy spaces and general $H^\al$ spaces. In Section~\ref{sec3}, we characterize the invariant and common invariant subspaces of $H^\al$ under multiplication by finite Blaschke factors. This is equivalent to characterizing invariance under the algebras  $H^\infty(B)$  and $H_1^\infty(B)$, respectively. In Section~\ref{sec4}, we demonstrate that the descriptions of the invariant subspaces can be made sharper if $\alpha$ is taken to be a continuous rotationally symmetric norm. Here, the underlying Blaschke factor is the monomial $B(z)=z^n$. The proofs of the main results in this section rely on the decomposition of $H^\al$ in terms of the space $M_\al(z^n)$, defined as the closure of $H^\infty(z^n)$ in the $\al$ norm. Lastly, in Section~\ref{sec5}, we define the $n$-outer functions in Hardy spaces equipped with a continuous rotationally symmetric norm and obtain a factorization of $H^\al$ functions in terms of $n$-inner and $n$-outer functions. 

\section{Notation and Preliminary Results}\label{sec2}

Throughout the paper, $\D$ stands for the open unit disk in the complex plane and it's boundary, the unit circle, is denoted by $\T$. The classical Lebesgue space on the unit circle $L^p$, $1 \leq p < \infty$, is a Banach space under the norm 
\[ \|f\|_p = \left(\int _ \T \lvert f \rvert ^p dm\right)^\frac{1}{p}, \] 
where $dm$ is the normalized Lebesgue measure on $\T$. The space $L^\infty$ is a Banach space under the essential supremum norm  $$\|f\|_\infty = inf \{ \; M \; : \; m\{ e^{i\theta} : \lvert f(e^{i\theta})\rvert > M \} = 0 \}.$$
It is well known that any function $f \in L^p$ can be represented as a Fourier series $f(z) = $ $\sum\limits_{n= - \infty}^{\infty} a_n z^n$. The Fourier coefficients are given by $a_n=\int_{\T} f z^{-n} dm$. The Hardy space $H^p$ is defined as the norm-closure (weak*-closure if $p = \infty$) of the analytic polynomials in $L^p$. $H^p$ turns out to be a closed subspace of $L^p$ which can be expressed as: $$ H^p = \{ f \in L^p \; : \; a_n = \int_{\T} f z^{-n} dm= 0,\;  \text{for all} \; n < 0 \}.$$ 
The space $H^\infty = H^1\cap L^\infty$ is a Banach space under the essential supremum norm and it multiplies all $H^p$ spaces back into $H^p$. 
\par
A norm $\al$ on $L^\infty$ is called a $\|.\|_1$-dominating normalized gauge norm if:\\
(1) $\al(1) = 1$, \\ 
(2) $\al(\lvert f \rvert) = \al(f)$  for every $f \in L^\infty$,\\
(3) $\al(f) \geq \|f\|_1 $ for every $f \in L^\infty$.
\smallskip
\par
\noindent
Further, we say that a norm $\al$ is continuous if 
    $$\lim\limits_{ m(E) \rightarrow 0^+} \al( \chi_E) = 0,$$
where $\chi_E$ stands for a characteristic function of $E \subset \T$.\\
Furthermore, we can also define $\al$ norm over all the measurable functions $f$ on $\T$ as
    $$\al(f)  = \sup \{ \al(s) : s~\text{is a simple function},~0 \leq s \leq \lvert f \rvert  \}.$$
The space $\mathcal{L}^\al = \{ f : \T \rightarrow \C $ measurable such that $\al(f) < \infty $\} is a Banach space under the continuous $\|.\|_1$-dominating normalized gauge norm  $\al$. The $\al$-closure of $L^\infty$ is the general Lebesgue space $L^\al$. Some important facts about $L^\al$ and $\mathcal{L}^\al$ spaces established in \cite{chen2017general} are $L^\infty \subsetneq \mathcal{L}^\al$, and $L^\infty \subsetneq L^\al \subset \mathcal{L}^\al \subset L^1 $. The space $L^\infty$ multiplies $L^\al$ back into $L^\al$ and the following inequality is satisfied: $\al(fg) \leq \|f\|_\infty \al(g)$ for all $f \in L^\infty$ and $g \in L^\al$. In fact this inequality holds for all $g\in \mathcal{L}^{\al}$.

The $\al$-closure of $H^\infty$ is denoted by $H^\al$ and it is a closed subspace of $L^\al$. The general Hardy space $H^\al$ is a Banach space under the norm $\al$. We also have a simpler description of $H^\al$ proved in \cite{chen2017general}, which is $H^\al = H^1 \cap L^\al$. Further, we have an important relationship that $H^\infty \subsetneq H^\al \subset H^1$.

A special class of norms on $L^\infty$, called rotationally symmetric norms, are also $\|.\|_1$-dominating normalized gauge norms. We say that a norm $\al$ on $L^\infty$ is rotationally symmetric if
\par
\noindent
(1) $\al(1) = 1$, \\
(2) $\al(\lvert f \rvert) = \al(f)$ for every $f \in L^\infty$,\\
(3) $\al(f_w) = \al(f)$. Here $f_w(z) = f(\overline{w}z)$ for each $w \in \T$ and $f \in L^\infty$.
\smallskip
\par
\noindent
It has been proved in \cite{Chen2014LebesgueAH} that a rotationally symmetric norm $\al$ must satisfy $\|.\|_1 \leq \al(.) \leq \|.\|_\infty$. Therefore, these norms are also $\|.\|_1$-dominating normalized gauge norms and are a special case of the same. Besides the $p$-norms $( 1 \leq p < \infty)$, other important examples of $\|.\|_1$-dominating normalized gauge norms are presented in \cite{chen2016vector-valued}. 
\par
A Blaschke factor with $n$ zeros $\al_1, \al_2,\ldots,\al_n \in \D $ is defined as 
    $$B(z) = \prod\limits_{i=1}^{n} \dfrac{z - \al_i}{1 - \overline{\al_i} z}.$$
Without loss of generality, we assume that $\al_1 = 0$ (due to the conformal invariance of $H^2$). The operator of multiplication by $B(z)$ denoted by $T_B$ is an isometry on $H^\al$. This can be easily seen, since $\lvert B(z) \rvert = 1$ a.e on $\T$ and $\al(Bf)$ = $\al(\lvert Bf \rvert)$ = $\al(\lvert B \rvert \lvert f \rvert)$ = $\al(\lvert f \rvert)$ = $\al(f)$. A closed subspace $\mathcal{M}$ of $H^\al$ is invariant under the operator $T_B$ if $B \mathcal{M} \subset \mathcal{M}$. Moreover, $\mathcal{M}$  is invariant under any subalgebra $A$ of $H^\infty$, if $ a f \in \mathcal{M}$ whenever $a \in A$ and $f \in \mathcal{M}$. The set $H^\infty(B)$ stands for the set $ \{ foB : f \in H^\infty\}$ which turns out to be a subalgebra of $H^\infty$ generated by $B$. We shall show in Lemma~{\ref{lem2.5}} that a closed subspace $\mathcal{M}$ of $H^\al$ is invariant under the operator $T_B$ if and only if $\mathcal{M}$ is invariant under the subalgebra $H^\infty(B)$. The weak*-closure of the linear span of $\{ 1, B^2, B^3,\ldots\}$ is a subalgebra of $H^\infty$ and is denoted by $H^\infty_1(B)$. It can be shown that $\mathcal{M}$ is invariant under $H^\infty_1(B)$ if and only if $\mathcal{M}$ is invariant under both $T_{B^2}$ and $T_{B^3}$.
\par
Put $B_j(z)$ =  $\prod\limits_{i=1}^{j} \dfrac{z - \alpha_i}{1 - \overline{\alpha_i} z}$. In \cite{singh1997multiplication}, it has been proved that the collection $\left \{ e_{jm} = \dfrac{\sqrt{1-{\lvert \al_{j+1} \rvert}^2}}{1 - \overline{\al_{j+1}} z} B_j B^m : 0 \leq j \leq n-1, m = 0,1,2, \ldots  \right \}$ is an orthonormal basis for $H^2$, and that $H^2$ can be decomposed as an orthogonal direct sum 
\begin{equation}\label{eq2.1}
H^2 = e_{00} H^2(B) \oplus e_{10} H^2(B) \oplus \cdots \oplus e_{n-1,0}  H^2(B)
\end{equation}
where $H^2(B)$ is the $\|.\|_2$-closure of all the analytic polynomials in $B$. A similar decomposition has also been proved for $H^p$ spaces \cite{sahni2012lax}. Formally, for $1 \leq p \leq \infty$, $H^p$ is the algebraic direct sum $e_{00} H^p(B) \oplus e_{10} H^p(B) \oplus\cdots\oplus e_{n-1,0} H^p(B)$. Here $H^p(B)$ is the $\|.\|_p$-closure of all the analytic polynomials in $B$.
\par
The invariance under multiplication by $B(z)$ is characterized in terms of $B$-inner functions which are more general than inner functions. Every inner function is a $B$-inner function for $B(z)=z$. An $H^\infty$ function $\phi$ is called $B$-inner if $\{ B^m \varphi : m = 0,1,2, \ldots \}$ is an orthonormal set in $H^2$. In general, for an $r$ tuple $(\varphi_1 ,\ldots , \varphi_r )$ of $H^\infty$ functions, we can speak of a $B$-inner matrix $A$ which is constructed as follows:\\
For each $1 \leq j \leq r$, we can write $\varphi_j$ as
$\varphi_j = \sum\limits_{i=0}^{n-1} e_{i0} \; \varphi_{ij}$ where $\varphi_{ij} \in H^2(B)$. The $n \times r$ matrix given by $A=(\varphi_{ij})_{n \times r}$ is a $B$-inner matrix if $A^* A = I$. In \cite{singh1997multiplication}, the condition $A^* A = I$ has been shown to be equivalent to the fact that the set $\{B^m \varphi_i : 1 \leq i \leq r ,  m = 0,1,2, \ldots \}$ is orthonormal in $H^2$. This is in line with the characterization of classical inner functions and usually provides a convenient way of testing the $B$-inner condition. When $B(z)= z^n$, we shall call $B$-inner functions simply as $n$-inner functions.
\par
Coming back to the basics of $L^\al$, the dual norm $\al^\prime$ of $\al$ on $L^\infty$ is defined as:
\begin{center}
    $\al^{\prime}(h) = \sup \Big\{ \Big \lvert \int\limits_{\T} f h dm  \Big \rvert$ : $f \in L^\infty$ and $ \al(f) \leq 1 \Big\}$  \\ \quad \quad \; 
 $= \sup \Big\{ \int\limits_{\T} \lvert f h \rvert dm $ : $f \in L^\infty$ and $ \al(f) \leq 1 \Big\}$.
\end{center}
It has been shown that if $\alpha$ is a $\|.\|_1$-dominating normalized gauge norm, or a rotationally symmetric norm, then ${\al}^{\prime}$ is also a $\|.\|_1$-dominating normalized gauge norm\cite{chen2017general}, or a rotationally symmetric norm \cite{Chen2014LebesgueAH}, respectively. We can now define the space $\mathcal{L}^{\al^{\prime}}$ along the same lines as $\mathcal{L}^{\al}$. The dual space of $L^\alpha$ can be identified with $\mathcal{L}^{\al^{\prime}}$. The precise formulation of an isometric isomorphism is given in the proposition below:
\begin{prop}[Chen  ~\cite{chen2017general}]\label{prop2.1}
Let $\al$ be a continuous $\|.\|_1$-dominating normalized gauge norm (or, a continuous rotationally symmetric norm) and let $\al^\prime$ be the dual norm of $\al$. Then $(L^{\al})^{\#}$ = $\mathcal{L} ^{\al^{\prime}}$,i.e., for every $ \phi \in (L^\al)^{\#}$, there is an $h \in \mathcal{L}^{\al^{\prime}}$ such that $\| \phi \| = \al^{\prime}(h)$ and 
$$ \phi(f)= \int_{\T} f h dm $$ 
for all $f \in L^\al$. 
\end{prop}

\begin{prop}[Chen~\cite{chen2017general}]\label{prop2.2}
Let $\al$ be a continuous $\|.\|_1$-dominating normalized gauge norm on $\mathcal{L}^\al$ and if $0 \leq f_1\leq f_2 \leq \ldots $ and $f_n \rightarrow f \; a.e.$. Then, $ \al(f_n) \rightarrow \al(f)$.
\end{prop} 

\begin{rem}
In particular taking $\al = \|.\|_1$, we get the classical monotone convergence theorem. 
\end{rem}
The lemma below states a version of Holder's inequality valid on $L^\alpha$. Its proof relies on the generalized monotone convergence theorem stated above (Proposition \ref{prop2.2}).

\begin{lem}\label{lem2.4}
If $f\in L^\al$ and $h\in\mathcal{L}^{\al^{\prime}}$, then $fh\in L^1$. Further, $\|fh\|_1 \leq \al(f) \al^{\prime}(h)$.  
\end{lem} 

\begin{proof}
The proof is easy when $f$ and $h$ both belong to $L^\infty$. This can be seen through the following calculation:

\begin{equation}\label{eq2.2}
\left \| \frac{f}{\al(f)} \; h \right\|_1 = \frac{1}{\al(f)} \int_{\T}  \lvert f h \rvert dm \leq \al^{\prime}(h).
\end{equation}
The inequality in ({\ref{eq2.2}}) follows by the definition of $\al^\prime$.
It can also be re-phrased as $ \| f h \|_1 \leq \al(f) \; \al^{\prime}(h)$.

We now show the validity of inequality (\ref{eq2.2}) for any $ f \in L^\al$ and $h \in \mathcal{L}^{\al^{\prime}}$. Since $\lvert f \rvert$ and $\lvert h \rvert$ are measurable functions on $\T$, there exist increasing sequences of non-negative simple functions ${s_n}$ and ${r_n}$ respectively, such that $s_n \to \lvert f \rvert \; a.e.$ and $r_n \to \lvert h \rvert \; a.e.$ By Proposition~{\ref{prop2.2}}, we have $\al(s_n) \to \al(\lvert f \rvert)$ and $\al^{\prime}(r_n) \to \al^{\prime}(\lvert h \rvert)$ as $n \to \infty$. Thus, 
$$\al(s_n) \; \al^{\prime}(r_n) \to \al(\lvert f \rvert) \; \al^{\prime}(\lvert h \rvert) \; \; \; \;  \text{as} \; \; n \to \infty.$$
Further, note that the sequence $s_n r_n \to \lvert fh \rvert \; a.e.$ So, by the classical monotone convergence theorem, we have 
    $$\|s_n r_n\|_1 \to \|fh\|_1 \; \; \;  as \; \; n \to \infty.$$
Applying the inequality (\ref{eq2.2}), we get 
\[  \|s_n r_n\|_1 \;  \leq \; \al(s_n) \; \al^{\prime}(r_n).\]
Finally, taking limit as $n\to\infty$, we conclude that for every $f \in L^\al$ and $h \in \mathcal{L^{\al^{\prime}}}$:
\[ \|fh\|_1 \; \leq \al(\lvert f \rvert) \; \al^{\prime}(\lvert h \rvert) = \al(f) \; \al^{\prime}(h). \]
\end{proof}

\begin{lem}\label{lem2.5}
Let $\al$ be a continuous $\|.\|_1$-dominating gauge norm. Suppose $\mathcal{M}$ is a closed subspace of $H^\al$. Then, $\mathcal{M}$ is invariant under $T_B$ if and only if $\mathcal{M}$ is invariant under the algebra $H^\infty(B)$. 
\end{lem}

\begin{proof}
The condition $H^\infty(B) \mathcal{M} \subset \mathcal{M}$ trivially implies the invariance of $\mathcal{M}$ under the operator $T_B$. We now prove the converse.

Suppose $T_B(\mathcal{M})\subset \mathcal{M}$. Then, $(T_B)^n \mathcal{M} \subset \mathcal{M}$ for all $n = 1,2,3, \ldots$ This implies that $\mathcal{M}$ is invariant under multiplication by all polynomials in $B$. We shall show that $\mathcal{M}$ is also invariant under multiplication by all functions in $H^\infty(B)$.
Let $k \in H^\infty(B)$, and $h\in \mathcal{M} $ be arbitrary. We can find a sequence of Cesaro means $\{\sigma_n(B)\}$ converging to $k$ in the weak* topology. That is, for every $f \in L^1$:
\begin{equation}\label{eq2.3}
\int_{\T} \sigma_n(B) f~dm \longrightarrow  \int_{\T} k f~dm \; \; \; \text{as} \; \; n \rightarrow \infty.
\end{equation}
In view of Lemma~{\ref{lem2.4}} and equation (\ref{eq2.3}) we conclude that for any non-zero element $u \in \mathcal{L}^{\al^\prime}$: 
\[\int_{\T} \sigma_n(B) h u~dm \longrightarrow  \int_
{\T} k h u~dm  \; \; \; \text{as} \; \; n \rightarrow \infty.
\]
Note that $\sigma_n(B)$ are polynomials in $B$, and hence $\sigma_n(B) h \in \mathcal{M} \subset L^\al $. Looking at the general form of any bounded linear functional on $L^\alpha$ (Proposition~{\ref{prop2.1}}), it follows that $\sigma_n(B) h$ $\rightarrow$ $kh$ weakly in $L^\al$. Because all closed subspaces are also weakly closed,  we infer that $kh \in \mathcal{M}$. This completes the proof.  
\end{proof}

\begin{rem}
The invariance of $\mathcal{M}$ under $T_{B^2}$ and $T_{B^3}$ but not under $T_B$ is equivalent to the invariance of $\mathcal{M}$ under the algebra $H^\infty_1(B)$ but not under $H^\infty(B)$. 
\end{rem}

We shall now state a few lemmas without proof which play a central role in proving the main results of this paper. 

\begin{lem}[Gamelin~\cite{gamelin1969uniform}]\label{lem2.7}
Suppose $\{f_n\}$ is a sequence of $H^\infty$ functions which converges to a $H^\infty$ function $f$ in the $H^2$ norm. Then, there exists a sequence $\{h_n\}$ of uniformly bounded $H^\infty$ functions which converges to 1 $a.e.$, such that $h_n f_n \to f$ in the weak* topology.
\end{lem}

\begin{lem}[Sahni and Singh~\cite{sahni2012lax}]\label{lem2.8}
Suppose $f \in H^\infty$ such that $f$ = $\varphi_1 g_1 + \cdots + \varphi_r g_r$ for some $g_1,\ldots,g_r$ $\in H^2(B)$, and for some $\varphi_1 ,\ldots, \varphi_r$ $B$-inner functions. Then $g_1,\ldots,g_r$ belongs to  $H^\infty(B)$. 
\end{lem}

\begin{lem}[Lance and Stessin~\cite{Lance1997multiplication}]\label{lem2.9}
Let $\varphi_1,\ldots,\varphi_r$ ($r \leq n$) are $B$-inner functions and for any $f\in H^\infty$ such that $f(z) = \sum_{j=1}^{r} \varphi_j(z) f_j(B(z))$. Then there exist constants $D_{j,p}$, such that $\|f_j\|_p \leq D_{j,p} \|f\|_p$ for each $j=1,\ldots,r$ and $1 \leq p \leq \infty$.
\end{lem}

\section{Generalization of Invariant Subspace Theorem in $H^\al$ space}\label{sec3}
In \cite{sahni2012lax}, the closed subspaces of $H^p$ for $0 < p \leq \infty$, that are invariant under multiplication by all powers of $B$, except the first power has been characterized. We extend the result for a far more general Hardy spaces $H^\al$, where $\al$ is $\|.\|_1$-dominating normalized gauge norm. Before stating our main result of this section, it is convenient to first present the proof of the invariance theorem explicitly for $H^\infty$ space. The idea of the proof is new and fairly different as seen in \cite{sahni2012lax}.

We first prove the following decomposition for $H^\infty$:
\begin{equation}\label{eq3.1}
H^\infty = e_{00} H^\infty(B) \oplus e_{10} H^\infty(B) \oplus \cdots \oplus e_{n-1,0}  H^\infty(B).
\end{equation}
Note that 
\begin{equation}\label{eq3.2}
    H^\infty=H^2\cap H^\infty = (e_{00} H^2(B) \oplus e_{10} H^2(B) \oplus \cdots \oplus e_{n-1,0}  H^2(B))\cap H^\infty.
\end{equation}
We will establish that the right hand sides of (\ref{eq3.1}) and (\ref{eq3.2}) coincide.\\
It is trivial to see that the right hand side of (\ref{eq3.1}) is contained in the right hand side of (\ref{eq3.2}). To establish the reverse containment, take any $f \in  \left ( e_{00} H^2(B) \oplus e_{10} H^2(B)  \oplus\cdots\oplus e_{n-1,0}  H^2(B) \right) \cap H^\infty$.
 This means that $f \in H^\infty$ and $f = e_{00} f_1 + e_{10} f_2 +\cdots+ e_{n-1,0}  f_n$ for some $f_1,\ldots, f_n \in H^2(B)$. Now in view of Lemma~\ref{lem2.8}, we find that $f_1,\ldots,f_n \in H^\infty(B)$. This completes the argument.

The proofs of Theorem~\ref{theo3.3} and Theorem~\ref{theo3.4} shall make use of the descriptions of common invariant subspaces established in \cite{sahni2012lax} and the invariant subspaces established in \cite{singh1997multiplication}. For the reader's convenience, we record both these results as under:
\begin{thm}[Sahni and Singh~\cite{sahni2012lax}]\label{theo3.1}
Let $\mathcal{M}$ be a closed subspace of $H^2$ invariant under ${H_1}^\infty(B)$ but not under $H^\infty(B)$. Then there exist $B$-inner functions $J_1,\ldots,J_r$ ($r \leq n$) such that
\begin{center}
     \[\mathcal{M} = \left (\sum_{i=1}^{k} \oplus \langle \varphi_i \rangle \right ) \oplus \sum_{j=1}^{r} \oplus B^2 J_j H^2(B)\]
\end{center}
where $k \leq 2r-1$, and $\varphi_i$ = $(\beta_{1i} + \beta_{2i}B)J_1 + (\beta_{3i} + \beta_{4i}B)J_2  + \ldots + (\beta_{2r-1,i} + \beta_{2r,i}B)J_r$,  for all i = 1,\ldots,k. Moreover, $\lvert\beta_{1i}\rvert^2 + \lvert\beta_{2i}\rvert^2 + \ldots + \lvert\beta_{2r,i}\rvert^2 = 1$.
\end{thm}

\begin{thm}[Singh and Thukral~\cite{singh1997multiplication}]\label{theo3.2}
Suppose $\mathcal{M}$ is a closed subspace of $H^2$ invariant under $H^\infty(B)$. Then there exist $B$-inner functions $J_1,\ldots,J_r$ ($r \leq n$) such that
$$ \mathcal{M} = J_1 H^2(B) \oplus J_2 H^2(B) \oplus\cdots\oplus J_r H^2(B). $$
Also the above representation is unique in the following sense: Suppose 
\[ \mathcal{M} = I_1 H^2(B) \oplus I_2 H^2(B) \oplus\cdots\oplus I_s H^2(B) \]
then $r =s$, and each $J_i = \sum\limits_{k=0}^{r} \phi_{ik} I_k$ for some scalars $\phi_{ik}$ such that the matrix $(\phi_{ik})$ is unitary. 
\end{thm}

Now we state and prove the common invariant subspace characterization for $H^\infty$.

\begin{thm}\label{theo3.3}
Let $\mathcal{M}$ be a weak*-closed subspace of $H^\infty$ such that $\mathcal{M}$ is invariant under ${H_1}^\infty(B)$, but not under $H^\infty(B)$. Then there exist $B$-inner functions $J_1,\ldots,J_r$ ($r \leq n$) such that
\begin{center}
     \[ \mathcal{M} = \left (\sum_{i=1}^{k} \oplus \langle \varphi_i \rangle \right ) \oplus \sum_{j=1}^{r} \oplus B^2 J_j H^\infty(B),\]
\end{center}
where $k \leq 2r-1$, and $\varphi_i$ = $(\beta_{1i} + \beta_{2i}B)J_1 + (\beta_{3i} + \beta_{4i}B)J_2  +\ldots + (\beta_{2r-1,i} + \beta_{2r,i}B)J_r$, for all i = 1,\ldots,k. We shall also see that $\beta_{1i}, \beta_{2i},\ldots, \beta_{2r,i}$ are complex numbers with $\lvert \beta_{1i}\rvert^2 + \lvert \beta_{2i}\rvert^2 +\cdots+ \lvert\beta_{2r,i}\rvert^2 = 1$.
\end{thm}

\begin{proof}
We first show that ${\overline{\mathcal{M}}}^{\|.\|_2} \cap H^\infty = \mathcal{M}$. Note that $\mathcal{M} \subset {\overline{\mathcal{M}}}^{\|.\|_2} $ and hence, $ \mathcal{M} \subset {\overline{\mathcal{M}}}^{\|.\|_2} \cap H^\infty$. 

Conversely, start with any $f \in {\overline{\mathcal{M}}}^{\|.\|_2} \cap H^\infty$. So, we can find a sequence $\{f_k\}$ in $\mathcal{M}$ such that $f_k \longrightarrow f$ in the norm of $H^2$. Here $\{f_k\}$ and $f$ are functions of $H^\infty$. Using the decomposition $H^\infty = e_{00} H^\infty(B) \oplus e_{10} H^\infty(B) \oplus\cdots\oplus e_{n-1,0}  H^\infty(B)$, we can write
\begin{equation}\label{eq3.3}
f_{k} = e_{00} {f_k}^{(1)} + e_{10} {f_k}^{(2)} +\cdots+ e_{n-1,0}  {f_k}^{(n)},
\end{equation}
\begin{equation}\label{eq3.4}
 f = e_{00} {f}^{(1)} + e_{10} {f}^{(2)} +\cdots+ e_{n-1,0} {f}^{(n)},
\end{equation}
where ${f_k}^{(i)}$ and $f^{(i)}$ belong to $H^\infty(B)$.\\ 
The following computation leads to the conclusion that $f_k^{(i)} \to f^{(i)}$ in $H^2$:
\begin{equation}\label{eq3.5}
\begin{aligned}
{\|f_k - f\|_2}^2 &= {\|e_{00} (f_k^{(1)} - f^{(1)}) + e_{10} (f_k^{(2)} - f^{(2)})+\cdots+ e_{n-1,0}  (f_k^{(n)} - f^{(n)})\|_2}^2  \\
&= {\|f_k^{(1)} - f^{(1)}\|_2}^2 + {\|f_k^{(2)} - f^{(2)}\|_2}^2 +\cdots+ {\|f_k^{(n)} - f^{(n)}\|_2}^2.\\ 
\end{aligned}
\end{equation}

By Lemma~\ref{lem2.7}, there exists a sequence $\{{h_k}^{(i)}\}$ of uniformly bounded $H^\infty$ functions that converges to 1 $a.e.$ such that $h_k^{(i)} f_k^{(i)}$ $\longrightarrow$ $f^{(i)}$ in the weak* topology.\\ 
Define
\begin{center}
    $g_k$ = $h_k^{(1)}(B^2) h_k^{(2)}(B^2) \ldots h_k^{(n)}(B^2)$.
\end{center}
Observe that $g_k \in H^\infty(B^2)$ and $h_k^{(i)}\to 1 \; a.e.$ as $k \to \infty$. Thus, $h_k^{(i)}(B^2)$ is a uniformly bounded sequence which converges to 1 $a.e.$ (because the composition operator on $H^\infty$ induced by $B^2$ is a contraction ,i.e, $\|foB^2\|_\infty \leq \|f\|_\infty$).
This forces $g_k\to 1 \; a.e.$\\
For each $i$, $f_k^{(i)} \to f^{(i)}$ in $H^2$, there exists subsequence $f_{k_j^{(i)}} \to f^{(i)} \; a.e.$ Without loss of generality, assume that $f_k^{(i)} \to f^{(i)} \;a.e$. This gives $\{g_k f_k^{(i)}\}$ converges to $f^{(i)} \;a.e$.\\
For any arbitrary $u \in L^1$, $\{g_k f_k^{(i)} u \}$ $\longrightarrow$ $f^{(i)} u \; a.e.$ and
\begin{center}
    $\lvert g_k f_k^{(i)} u - f^{(i)} u + f^{(i)} u \rvert$ $\leq$ $\lvert g_k f_k^{(i)} u - f^{(i)} u \rvert + \lvert f^{(i)} u \rvert $.\\
\end{center}
In view of the convergence of $\{g_k f_k^{(i)} u \}$, we can find a natural number $k_0$ such that
$$\lvert g_k f_k^{(i)} u \rvert  < 1 + \lvert f^{(i)}u \rvert \; \; a.e. \; \; \text{for all} \; \;  k \ge k_0.$$
By the Dominated Convergence theorem,
\begin{center}
    $\int_{\T} g_k f_k^{(i)} u dm$ $\longrightarrow$ $\int_{\T} f^{(i)} u dm $.
\end{center}
Since $u$ was an arbitrary element of $L^1$, so $\{g_k f_k^{(i)}\}$ converges to $f^{(i)}$ for fixed $i = 1,\ldots,n$ in the weak* topology.\\
As $e_{i0} u \in L^1$ whenever $u \in L^1$,  $\{e_{i0} g_k f_k\}$ $\longrightarrow$ $e_{i0} f^{(i)}$ in the weak* topology.\\
Therefore, \[ \sum_{i=0}^{n-1} e_{i0} g_k f_k^{(i+1)} \; \text{converges weak}^* \; \text{to} \; \sum_{i=0}^{n-1} e_{i0} f^{(i+1)} .\]
This implies that $\{{g_k} {f_k}\}$ converges weak* to $f$. Since ${g_k} \in H^\infty(B^2)$ and $\mathcal{M}$ is invariant under the operator $T_{B^2}$, we have $\{{g_k} {f_k}\}\in \mathcal{M}$. Further, since $\mathcal{M}$ is weak*-closed, $f \in \mathcal{M}$ and hence ${\overline{\mathcal{M}}}^{\|.\|_2} \cap H^\infty \subset \mathcal{M}$.

Observe that $\overline{\mathcal{M}}^{\|.\|_2}$ is a closed subspace of $H^2$ which is invariant under $H_1^\infty(B)$, but not under $H^\infty(B)$. So by Theorem~\ref{theo3.1}, we can write
\begin{equation}\label{eq3.6}
   \overline{\mathcal{M}}^{\|.\|_2} = \left (\sum_{i=1}^{k} \oplus \langle \varphi_i \rangle \right ) \oplus \sum_{j=1}^{r} \oplus B^2 J_j H^2(B)
\end{equation}
where $k \leq 2r-1$, and $\varphi_i$ = $(\beta_{1i} + \beta_{2i}B)J_1 + (\beta_{3i} + \beta_{4i}B)J_2  +  ...  + (\beta_{2r-1,i} + \beta_{2r,i}B)J_r$,  for all i = 1,2,...,k.\\
Therefore, we have 
\[ \mathcal{M} =\overline{\mathcal{M}}^{\|.\|_2} \cap H^\infty = \left(\left (\sum_{i=1}^{k} \oplus \langle \varphi_i \rangle \right ) \oplus \sum_{j=1}^{r} \oplus B^2 J_j H^2(B) \right) \cap H^\infty.\]
\par
It remains to be shown that $\mathcal{M} = \left (\sum_{i=1}^{k} \oplus \langle \varphi_i \rangle \right ) \oplus \sum_{j=1}^{r} \oplus B^2 J_j H^\infty(B)$.\\
Note that
\begin{center}
    $H^\infty \cap \left[\left (\sum_{i=1}^{k} \oplus \langle \varphi_i \rangle \right ) \oplus \sum_{j=1}^{r} \oplus B^2 J_j H^\infty(B)\right]$ $\subset$ $H^\infty \cap \left[\left( (\sum_{i=1}^{k} \oplus \langle \varphi_i \rangle \right ) \oplus \sum_{j=1}^{r} \oplus B^2 J_j H^2(B)\right]$.
\end{center}
Therefore, 
\begin{equation}\label{eq3.7}
    \left (\sum_{i=1}^{k} \oplus \langle \varphi_i \rangle \right ) \oplus \sum_{j=1}^{r} \oplus B^2 J_j H^\infty(B) \subset \mathcal{M}.
\end{equation}
For the reverse inclusion, write an arbitrary $f \in \mathcal{M}$ as 
\begin{center}
    $f$ = $\left (\sum_{i=1}^{k} a_i \varphi_i  \right ) + \sum_{j=1}^{r} B^2 J_j f_j$ 
\end{center}
where $a_i$'s are complex numbers and $f_1 ,\ldots, f_r$ $\in H^2(B)$.\\
Since the function $ h = \sum_{j=1}^{r} B^2 J_j f_j \in H^\infty$ and $J_1,\ldots,J_r$ are $B$-inner, so by Lemma~\ref{lem2.8}, the functions $B^2f_1,\ldots,B^2f_r$ are in $H^\infty(B)$. Hence $f_1,\ldots,f_r$ are also in $H^\infty(B)$.
This forces $f \in \left (\sum_{i=1}^{k} \oplus \langle \varphi_i \rangle \right ) \oplus \sum_{j=1}^{r} \oplus B^2 J_j H^\infty(B)$. This completes the proof.
\end{proof} 

Now, we obtain the $B$-invariant subspaces for $H^\infty$ using elementary arguments which completely avoids the general inner-outer factorization of $H^p$ functions in terms of $B$-inner functions and outer functions. For details, the reader may refer to Proposition $6$ in  \cite{Lance1997multiplication}. Our proof is on similar lines as the proof of the decomposition (\ref{eq3.1}) and is even more elementary than the proof presented in \cite{sahni2012lax}.

\begin{thm}\label{theo3.4}
Let $\mathcal{M}$ be a weak*-closed subspace of $H^\infty$ such that $\mathcal{M}$ is invariant under $T_B$. Then there exist $B$-inner functions $J_1,\ldots,J_r$ ($r \leq n$) such that
     \[\mathcal{M} = J_1 H^\infty(B) \oplus\cdots\oplus J_r H^\infty(B).\]
\end{thm}
\begin{proof}
Using the arguments similar to those in the proof of Theorem~\ref{theo3.3}, we see that $\mathcal{M} = {\overline{\mathcal{M}}}^{\|.\|_2} \cap H^\infty$. Again, since  ${\overline{\mathcal{M}}}^{\|.\|_2}$ is a $T_B$-invariant subspace of $H^2$, so by Theorem~\ref{theo3.2} we can write:
    \[\overline{\mathcal{M}}^{\|.\|_2} = J_1 H^2(B) \oplus\cdots\oplus J_r H^2(B)\]
for some $B$-inner functions $J_1,\ldots,J_r$ ( $r \le n$).

This implies
\[\mathcal{M} =\overline{\mathcal{M}}^{\|.\|_2} \cap H^\infty = \left ( J_1 H^2(B) \oplus\cdots\oplus J_r H^2(B) \right ) \cap H^\infty. \]
Lastly, we can show that $\left ( J_1 H^2(B) \oplus\cdots\oplus J_r H^2(B) \right ) \cap H^\infty=J_1 H^\infty(B) \oplus\cdots\oplus J_r H^\infty(B)$ on the same lines as in the proof of the decomposition in equation (\ref{eq3.1}).
\end{proof}

For the convenience of the reader, we now state the main results of this paper which generalize the Lax-Halmos type results in \cite{sahni2012lax}. Before proving these, we discuss the road map of the proof briefly.

\begin{thm}\label{theo3.5}
Let $\al$ be a continuous $\|.\|_1$-dominating normalized gauge norm and $\mathcal{M}$ be an $\al$-closed linear subspace of $H^\al$. Suppose $\mathcal{M}$ is invariant under $T_{B^2}$ and $T_{B^3}$, but not under $T_B$. Then, 
\[  \mathcal{M} = \left (\sum_{i=1}^{k} \oplus \langle \varphi_i \rangle \right ) \oplus B^2 \; [(J_1 M_\al(B) \oplus\cdots\oplus J_r M_\al(B))]_\al \]
where $M_\al(B) := [H^\infty(B)]_\al$, $J_1,\ldots,J_r$ ($r \leq n$) are $B$-inner functions and $k \leq 2r-1$. Furthermore, $\varphi_i$ = $(\beta_{1i} + \beta_{2i}B)J_1 + (\beta_{3i} + \beta_{4i}B)J_2  + \ldots+ (\beta_{2r-1,i} + \beta_{2r,i}B)J_r$ for $i=1,\ldots,k$. The complex numbers $\beta_{ji}$ satisfy $\lvert \beta_{1i}\rvert^2 + \lvert\beta_{2i}\rvert^2 +\cdots+ \lvert\beta_{2r,i}\rvert^2 = 1$.\\
Here $[H^\infty(B)]_\al$ stands for the $\al$-closure of $H^\infty(B)$.
\end{thm}
\begin{thm}\label{theo3.6}
Suppose $\al$ to be a continuous $\|.\|_1$-dominating normalized gauge norm and $\mathcal{M}$ be an $\al$-closed linear subspace of $H^\al$. Suppose $\mathcal{M}$ is invariant under $T_{B}$. Then, 
\[ \mathcal{M} = [J_1 M_\al(B) \oplus\cdots\oplus J_r M_\al(B)]_\al \] for some $B-$inner functions $J_1,...,J_r$ $(r \leq n)$.
\end{thm}

The proof of Theorem~\ref{theo3.5} comprises of the following key steps: First, we identify $\mathcal{M}\cap H^\infty$ to be a non-trivial weak*-closed subspace of $H^\infty$ which has the same invariance properties as that of $\mathcal{M}$. Next, we use Theorems~\ref{theo3.3} and \ref{theo3.4} to describe $\mathcal{M}\cap H^\infty$. Lastly, we finish off the proof by establishing that $\mathcal{M}\cap H^\infty$ is $\al$-dense in $M$. It is worth pointing out that our arguments to show the $\al$-denseness of $\mathcal{M}\cap H^\infty$ do not rely on any factorization of $L^\alpha$ functions as is the case in \cite{chen2017general}. Instead, we work with the classical inner-outer factorization of $H^2$ functions.

\par
The following two lemmas will be needed to prove that $\mathcal{M} \cap H^\infty$ is a weak*-closed subspace of $H^\infty$ (Lemma \ref{lem3.9}). 

\begin{lem}[Krein-Smulian Theorem, \cite{conway2019course}]\label{lem3.7}
Let $X$ be a Banach space. A convex set in $X^{\#}$ is weak*-closed if and only if its intersection with the closed unit ball $\{ \phi \in X^{\#} : \|\phi\| \leq 1 \}$ of $X^{\#}$ is weak*-closed. \\
\end{lem}
\begin{lem}[Chen~\cite{chen2017general}]\label{lem3.8}
Let $\al$ be a continuous $\|.\|_1$-dominating normalized gauge norm. Let $\mathbb{B}(L^\infty) = \{ f \in L^\infty : \|f\|_\infty \leq 1 \}$ denote the closed unit ball on $L^\infty$. Then\\
$(1)$ The $\al$-topology and the $\|.\|_2$-topology coincide on $\mathbb{B}(L^\infty)$.\\
$(2)$ The closed unit ball $\mathbb{B}(L^\infty)$ is $\al$-closed.
\end{lem}
 
\begin{lem}\label{lem3.9}
Suppose $\al$ is a continuous $\|.\|_1$-dominating normalized gauge norm and $\mathcal{M}$ be a closed subspace of $H^\al$. Then $ \mathcal{M} \cap H^\infty $ is weak*-closed in $H^\infty$.
\end{lem}
 \begin{proof}
Let $\mathbb{B} = \{ \phi \in (L^1)^{\#} : \|\phi\| \leq 1 \}$ be the closed unit ball in $(L^1)^{\#}$. Since $(L^1)^{\#}$ can be identified with $L^\infty$, so $\mathbb{B}$ can be identified with $\mathbb{B}(L^\infty) = \{ f \in L^\infty(\T) : \|f\|_\infty \leq 1\}$.\\
We shall thus consider $\mathcal{M} \cap H^\infty \cap \mathbb{B}(L^\infty)$. Note that  $\mathcal{M} \cap H^\infty \cap \mathbb{B}(L^\infty)$ simplifies to $\mathcal{M} \cap \mathbb{B}(L^\infty)$. This follows by observing that any $f\in \mathcal{M} \cap \mathbb{B}(L^\infty)$ also belongs to $H^1\cap L^\infty=H^\infty$.\\
Thus, we now only need to prove that $ \mathcal{M} \cap \mathbb{B}(L^\infty)$ is weak*-closed in $L^\infty$.\\
By Lemma~\ref{lem3.8}, we see  that $\mathcal{M} \cap \mathbb{B}(L^\infty)$ is closed in $L^\alpha$. This in turn implies that $ \mathcal{M} \cap \mathbb{B}(L^\infty)$ is closed in the $L^2$ norm.\\
Let $\{ f_\lambda \}$ be a net in $\mathcal{M} \cap \mathbb{B}(L^\infty)$ such that $ f_\lambda$ converges weak* to $f$. This means that for every $ h \in L^1$, $\int_{\T} (f_\lambda - f) h dm $ $\to$ 0 . In particular, $\int_{\T} (f_\lambda - f) h dm $ $\to$ 0, for every $ h \in L^2$. This implies $ f_\lambda$ $\to$ $f$ weakly in $L^2$. Since a closed convex set in a Banach space is weakly closed, so $\mathcal{M} \cap \mathbb{B}(L^\infty)$ is also weakly closed in $L^2$. Hence $f \in \mathcal{M} \cap \mathbb{B}(L^\infty)$.
\end{proof}

\begin{rem}
Note that the collection $\mathcal{M} \cap \mathbb{B}(L^\infty) \neq {[0]}$ whenever $\mathcal{M} \cap H^\infty \neq {[0]}$. In case $\mathcal{M}$ is invariant under any power of $B$, then we can guarantee that $\mathcal{M} \cap H^\infty \neq {[0]}$.
\end{rem}

\begin{lem}\label{lem3.11}
Let $\al$ be a continuous $\|.\|_1$-dominating normalized gauge norm and $\mathcal{M}$ be a closed subspace of $H^\al$ which is invariant under $T_B$. Then, $\mathcal{M} \cap H^\infty$ is $\al$-dense in $\mathcal{M}$.
\end{lem}
\begin{proof}
Since $\mathcal{M}$ is an $\al$-closed subspace of $H^\al$, it follows that $[{\mathcal{M} \cap H^\infty}]_\al$ $\subset$ $\mathcal{M}$.

Conversely, any $f \in \mathcal{M} \subset H^\al \subset H^1$ can be factorized as $f=IO$, where $I$ is an inner function and $O$ is an outer function. \\
We can further write $f$ as a product of three $H^2$ functions: 
\begin{center}
    $ f $ = $ I$ $ O^{\frac{1}{2}}$ $ O^{\frac{1}{2}}.$
\end{center}
For the sake of convenience, put $f_1=I$, $f_2=f_3=O^{\frac{1}{2}}$.\\
Using the decomposition of $H^2$ in equation (\ref{eq2.1}), we can express each $f_k$ as
\begin{equation}\label{eq3.8}
f_k = e_{00} g_1^{(k)} + e_{10} g_2^{(k)}+\cdots+ e_{n-1,0}  g_n^{(k)}
\end{equation}
where $g_1^{(k)}, g_2^{(k)},\ldots, g_{n}^{(k)} \in H^2(B)$. It is well known that the operator $T : H^2 \rightarrow H^2$, defined by $T(f) = f(B(z))$, is an isometry and its range is $H^2(B) = \{ foB : f \in H^2 \}$ (see \cite{cowen2019composition}). Hence, each $g_i^{(k)} \in H^2(B)$ can be written as $g_i^{(k)} = s_i^{(k)}(B(z))$ for some $s_i^{(k)}\in H^2$.\\
Thus equation (\ref{eq3.8}) becomes
\begin{equation}\label{eq3.9}
    f_k = e_{00} s_1^{(k)}(B(z)) + e_{10} s_2^{(k)}(B(z)) +\cdots+ e_{n-1,0}  s_n^{(k)}(B(z)).
\end{equation}
\vspace{0.25cm}
Define
\begin{center}
$q_m^{(ik)}(z)$ = $ exp \left( \dfrac{- \lvert s_i^{(k)}(z) \rvert^{\frac{1}{2}} \; - i \; (\lvert s_i^{(k)}(z) \rvert^{\frac{1}{2}})^{\sim}}{m} \right).$
\end{center}
Here $(\lvert s_i^{(k)}\rvert^{\frac{1}{2}})^{\sim}$ denotes for the harmonic conjugate of the harmonic extension of the boundary function $\lvert s_i^{(k)}(e^{i\theta})\rvert^{\frac{1}{2}}$ onto the unit disc (this is possible for all $L^p$ functions \cite{koosis1998introduction}). 

So, $\{q_m^{(ik)}(z)\}$ is a sequence of analytic functions such that $\lvert q_m^{(ik)}(z) \rvert \leq 1$ for each $m \ge 1$. Next, for each $k$, construct a sequence $h_m^{(k)}(z)$ = $q_m^{(1k)}(z)$ $q_m^{(2k)}(z)\ldots q_m^{(n k)}(z)$, then $\{h_m^{(k)}(z)\} \in H^\infty$ and 
\begin{equation*}
\hspace{-1cm} h_m^{(k)}(B(z)) f_k(z) = h_m^{(k)}(B(z)) \left (e_{00} s_1^{(k)}(B(z)) +\cdots+ e_{n-1,0}  s_{n}^{(k)}(B(z)) \right ) 
\end{equation*}
\[\hspace{2.5 cm} = e_{00} h_m^{(k)}(B(z)) s_1^{(k)}(B(z)) +\cdots+ e_{n-1,0}  h_m^{(k)}(B(z)) s_n^{(k)}(B(z)). \]
\smallskip
Moreover,\\
\smallskip
$\lvert h_m^{(k)}(B(z))f_k\rvert \leq \lvert q_m^{(1k)}(B(z)) \rvert \lvert s_1^{(k)}(B(z)) \rvert +\cdots+ \lvert q_m^{(nk)}(B(z))\rvert \lvert s_n^{(k)}(B(z))\rvert$\\
\smallskip
\hspace{1cm}$=exp \left ( \dfrac{- \lvert s_1^{(k)}(z)\rvert^{\frac{1}{2}}}{m} \right )\lvert s_1^{(k)}(B(z))\rvert+\cdots+ exp \left ( \dfrac{- \lvert s_n^{(k)}(z)\rvert^{\frac{1}{2}}}{m}\right)\lvert s_n^{(k)}(B(z))\rvert$\\
\smallskip
\hspace{1cm}$\leq C$ for $C$ to be any constant.\\
Hence $h_m^{(k)}(B(z)) f_k \in H^\infty$ for each $k = 1,2,3$.\\
\smallskip
Note that $h_m^{(1)}(B(z)) \; h_m^{(2)}(B(z)) \;h_m^{(3)}(B(z))$ $\to$ 1 $a.e.$
Therefore  $t_m(B(z)) := h_m^{(1)}(B(z)) \; h_m^{(2)}(B(z)) \;h_m^{(3)}(B(z)) \to 1 \; a.e.$ as $m \to \infty$, and $\lvert t_m(B(z)) - 1 \rvert \to 0 \; a.e.$, implies 
\begin{center}
    $\|t_m(B(z)) - 1\|_\infty$ $\longrightarrow$ 0 as $m \to \infty$.
\end{center}
We can see that $\al(t_m(B(z)) f - f)$ = $\al(f \; (t_m(B(z)) - 1))$ $\leq \al(f) \; \|t_m(B(z)) - 1\|_\infty $, i.e, 
\begin{center}
    $\al(t_m(B(z)) f - f)$ $\longrightarrow$ 0 as $m \to \infty$.
\end{center}
Since $h_m^{(k)}(B(z))f_k \in H^\infty$, so $ \{t_m(B(z)) f\}$ $= h_m^{(1)}(B(z)) f_1$ $h_m^{(2)}(B(z))f_2$ $h_m^{(3)}(B(z))f_3$ belongs to  $H^\infty$.\\ 
Lastly, the invariance of $\mathcal{M}$ under $H^\infty(B)$ gives $\{t_m(B(z)) f\}$ $\in \mathcal{M}$. So, there is a sequence $\{t_m(B(z)) f\}$ in $\mathcal{M} \cap H^\infty$ such that $\al(t_m(B(z)) f - f)$ $\to$ 0 as $m \to \infty$, which implies that $f \in [\mathcal{M} \cap H^\infty]_\al$. Therefore, $\mathcal{M}$ = $[\mathcal{M} \cap H^\infty]_\al$.
\end{proof}
We now revert to the proof of Theorem \ref{theo3.5}.
\begin{proof}[Proof of Theorem 3.5]
We first show that $ \mathcal{M} \cap H^\infty \neq [0]$. Let $ 0 \neq f \in \mathcal{M}$ then as argued previously in Lemma~\ref{lem3.11}, we can express $f$ as 
\begin{center}
    $f$ = $f_1$ $f_2$ $f_3$ 
\end{center}
for some $f_1, f_2, f_3 \in H^2$. We follow the same argument to express each $f_k$ as 
\begin{center}
    $ f_k = e_{00} s_1^{(k)}(B^2(z)) + e_{10} s_2^{(k)}(B^2(z)) +\cdots+ e_{2n-1,0} s_{2n}^{(k)}(B^2(z)) $
\end{center}
for some $s_1^{(k)}(B^2(z))$, $s_2^{(k)}(B^2(z))$,\ldots, $s_{2n}^{(k)}(B^2(z)) \in H^2(B^2)$. Define
\begin{center}
    $q_{ik}(z)$ = $ exp \left( \dfrac{- \lvert s_i^{(k)}(z) \rvert^{\frac{1}{2}} \; - i \; (\lvert s_i^{(k)}(z) \rvert^{\frac{1}{2}})^{\sim}}{2} \right)$
\end{center}
( $\sim$ denotes the harmonic conjugate and defined in a similar fashion as above).\\
\smallskip
Furthermore, $\lvert q_{ik}(z) \rvert \in H^\infty$ and $\lvert q_{ik}(z) \rvert \leq 1$.\\\smallskip
Define $h_k(z) = q_{1k}(z) \; q_{2k}(z) \ldots q_{2n,k}(z)  \in H^\infty$. Note that,
\\ \smallskip
$ h_k(B^2(z)) f_k(z) = h_k(B^2(z)) \left (e_{00} s_1^{(k)}(B^2(z))  +\cdots+ e_{2n-1,0} s_{2n}^{(k)}(B^2(z)) \right )$ 

\hspace{1.5cm}  $= e_{00} h_k(B^2(z)) s_1^{(k)}(B^2(z)) +\cdots+ e_{2n-1,0} h_k(B^2(z)) s_{2n}^{(k)}(B^2(z))$\\ \smallskip
and hence $ h_k(B^2(z)) f_k(z)\in H^\infty$ for each $k = 1,2,3$.\\ 
\smallskip
This implies $h_1(B^2(z))h_2(B^2(z))h_3(B^2(z)) f$ is in $H^\infty$. By Lemma~\ref{lem2.5} we see that $h_1(B^2(z)) h_2(B^2(z)) h_3(B^2(z)) f \in \mathcal{M} \cap H^\infty$. This establishes that $\mathcal{M} \cap H^\infty \neq [0]$ whenever $ \mathcal{M} \neq [0]$.
\par
In view of Lemma~\ref{lem3.9}, $ \mathcal{M} \cap H^\infty $ is a weak*-closed subspace of $H^\infty$ which is also invariant under $T_{B^2}$ and $T_{B^3}$, but not under $T_B$. By Theorem~\ref{theo3.3}, 
   \[ \mathcal{M} \cap H^\infty = \left (\sum_{i=1}^{k} \oplus \langle \varphi_i \rangle \right ) \oplus \sum_{j=1}^{r} \oplus B^2 J_j H^\infty(B)\]
where $k \leq 2r-1$, and it follows from the Lemma~\ref{lem3.7} that
\[ \mathcal{M} = \left [ \left (\sum_{i=1}^{k} \oplus \langle \varphi_i \rangle \right ) \oplus \sum_{j=1}^{r} \oplus B^2 J_j H^\infty(B) \right ]_{\al} \]
\[ \; \; \; = \left (\sum_{i=1}^{k} \oplus \langle \varphi_i \rangle \right ) \oplus \left [ \sum_{j=1}^{r} \oplus B^2 J_j H^\infty(B)
\right ]_{\al}\]
\begin{equation}\label{eq3.10}
    \hspace{2.6cm}  = \left (\sum_{i=1}^{k} \oplus \langle \varphi_i \rangle \right ) \oplus B^2 \; \left[(J_1 H^\infty(B) \oplus\cdots\oplus J_r H^\infty(B)) \right]_{\al}.
\end{equation}

We now claim that $[ (J_1 H^\infty(B) \oplus\cdots\oplus J_r H^\infty(B)) ]_\al$ = $[(J_1 M_\al(B) \oplus\cdots\oplus J_r M_\al(B))]_\al$.\\
\smallskip
It can easily be shown $[ (J_1 H^\infty(B) \oplus\cdots\oplus J_r H^\infty(B)) ]_\al$ $\subset$ $[(J_1 M_\al(B) \oplus\cdots\oplus J_r M_\al(B))]_\al$ because $H^\infty(B)$ $\subset$ $M_\al(B)$.\\
\smallskip
To show the other containment, consider $ f \in J_1 
M_\al(B) \oplus\cdots\oplus J_r M_\al(B)$, $\exists$ $ f_1,\ldots,f_r \in M_\al(B)$ such that $ f = J_1 f_1 +\cdots+ J_r f_r$. Since  $M_\al(B) = \left [H^\infty(B) \right]_{\al}$, there exist the sequence of functions $\{f_n^{(1)}\} ,\ldots, \{f_n^{(r)}\}$ in $H^\infty(B)$ such that
\[ \al( f_n^{(i)} - f_i) \longrightarrow 0 \; \; \text{ as} \; \;  n \to \infty\]
for $ i = 1, 2,\ldots, r$. \\
Because $J_i$ is $B$-inner function, $ \al( J_i f_n^{(i)} - J_i f_i)$ $\leq$ $\|J_i\|_\infty \; \al( f_n^{(i)} - f_i) $, we have
\[ \al( J_i f_n^{(i)} - J_i f_i) \longrightarrow 0 \; \; \text{ as} \; \; n \to \infty \]
which again gives \[J_1 f_n^{(1)} +\cdots+ J_r f_n^{(r)} \longrightarrow J_1 f_1 +\cdots+ J_r f_r \; \;  \text{in} \; \;  H^\al \; \; \text{as} \; \;  n \longrightarrow \infty.\]
Thus, $ f \in  \left [ (J_1 H^\infty(B) \oplus\cdots\oplus J_r H^\infty(B)) \right ]_{\al} $ and so $[J_1 
M_\al(B) \oplus\cdots\oplus J_r M_\al(B)]_\al$ $\subset$ $[ (J_1 H^\infty(B) \oplus\cdots\oplus J_r H^\infty(B)) ]_\al$.\\
Hence equation (\ref{eq3.10}) becomes:  
\[
    \mathcal{M} = \left (\sum_{i=1}^{k} \oplus \langle \varphi_i \rangle \right ) \oplus B^2 \; \left[ (J_1 M_\al(B) \oplus\cdots\oplus J_r M_\al(B)) \right]_{\al}.
\]
This completes the proof.
\end{proof} 
 
\begin{cor}
Let $\mathcal{M}$ be a closed subspace of $H^\al$ which is invariant under $H_1^\infty(z)$, but is not invariant under $H^\infty(z)$. Then there exist scalars $\beta_1$ and $\beta_2$ in $\mathbb{C}$ such that $|\beta_1|^2 + |\beta_2|^2 = 1$, $\beta_1 \neq 0$, and an inner function $J$ such that
\[ \mathcal{M} = (\beta_1 + \beta_2 z) J  \oplus z^2 J H^\al. \]
\end{cor}

\begin{proof}
Taking $B(z) = z$ in Theorem~\ref{theo3.5}, we get: 
\[ \mathcal{M} = (\beta_1 + \beta_2 z) J  \oplus \left [ z^2 J H^\al \right ]_\al, \] where $J$ is an inner function, and $\beta_1$ and $\beta_2$ are scalars in $\mathbb{C}$ such that $|\beta_1|^2 + |\beta_2|^2 = 1$. Since $z^2$ and $J$ are inner functions and the operator of multiplication by inner function on $H^\al$ acts as an isometry, therefore 
\[ \mathcal{M} = (\beta_1 + \beta_2 z) J  \oplus z^2 J H^\al. \]
Moreover, $\beta_1 \neq 0$ because if $\beta_1 = 0$ then $\mathcal{M}$ becomes invariant under $z$, which is clearly a contradiction.
\end{proof}

We now prove Theorem~\ref{theo3.6}, which generalizes the Beurling type result proved by Chen \cite{chen2017general}. The proof is on the similar lines as Theorem~\ref{theo3.6}, so we only provide a broad sketch. 
\begin{proof}[Proof of Theorem 3.6]
It can be shown as in the proof of Theorem~\ref{theo3.5} that $\mathcal{M} \cap H^\infty$ is non-trivial and invariant under $T_B$. Further, by Lemma~\ref{lem3.9} it follows that $\mathcal{M} \cap H^\infty$ is a weak*-closed subspace of $H^\infty$. Next, in view of Lemma~\ref{lem3.11} we conclude that $\mathcal{M} = [\mathcal{M} \cap H^\infty]_{\al}$. By Theorem~\ref{theo3.4}, there exist $B$-inner functions $J_1,\ldots,J_r$ (with $r\le n$) such that
\begin{equation}\label{eq3.11}
    \mathcal{M} \cap H^\infty = J_1 H^\infty(B) \oplus\cdots\oplus J_r H^\infty(B). 
 \end{equation}
Now, by taking $\al$-closure on both sides of equation~\ref{eq3.11}, 
\[ \mathcal{M} = \left [ J_1 H^\infty(B) \oplus\cdots\oplus J_r H^\infty(B) \right ]_{\al}.\]
The claim that $\left [ J_1 H^\infty(B) \oplus\cdots\oplus J_r H^\infty(B) \right]_{\al}$ = $\left [J_1 M_\al(B) \oplus\cdots\oplus J_r M_\al(B) \right ]_\al$ follows by the arguments identical to those in Theorem~\ref{theo3.5}. 
\end{proof}

\begin{cor}[Chen~\cite{chen2017general}]
Suppose $\mathcal{M}$ is a closed subspace of $H^\al$ such that $z \mathcal{M} \subset \mathcal{M}$. Then there exists an inner function $J$ such that $ \mathcal{M} = J H^\al$. 
\end{cor}

\begin{proof}
Take $B(z) = z$ in Theorem~\ref{theo3.6}, we have 
\[ \mathcal{M} = [J H^\al]_\al \] where $J$ is an inner function. Also, the multiplication by an inner function on $H^\al$ will be an isometry, which implies \[ \mathcal{M} = J H^\al.\]
\end{proof}

In the next section, we show that descriptions of the invariant subspaces of Theorem~\ref{theo3.5} and ~\ref{theo3.6} can be made sharper when $\al$ is a continuous rotationally symmetric norm.

\section{Case of Continuous Rotationally Symmetric Norms}\label{sec4}
For a finite Blaschke factor $B(z)$, it is well known that the classical Hardy spaces $H^p$ ($0<p\le\infty$) can be decomposed as:
\begin{center}
    $H^p = e_{00} H^p(B) \oplus e_{10} H^p(B) \oplus\cdots\oplus e_{n-1,0}  H^p(B).$ 
\end{center}
Recall that $H^p(B)$ denotes the closed linear span of $\{ B^m : m=0,1,2,\ldots\}$ and $\oplus$ is an algebraic direct sum. This direct sum becomes an orthogonal direct sum when $p= 2$.\\
Consider $f\in H^\al$, and since $f \in H^1$, we can write $f(z)$ as a Fourier series
   $$ f(z) = \sum_{j=0}^{\infty} \hat{f}(j) z^j $$
where $\hat{f}(j)$ = $\int_\T f(z) z^{-j} dm $ for $j = 0,1,2, \ldots$
\par
\smallskip
Consider the $n$th Cesaro means
\[ \sigma_n(f) = \frac{S_0(f) + S_1(f) +\cdots+ S_n(f)}{n+1} \] 
where $ S_k(f)$ = $\sum_{j=0}^{k} \hat{f}(j) z^j $ stands for the $k$th partial sum.\\
The following lemma shows that any function in $H^\al$ converge to it's Cesaro means in $\al$-norm. 
\begin{lem}[Chen~\cite{Chen2014LebesgueAH}]\label{lem4.1}
Let $\al$ be a continuous rotationally symmetric norm and $f \in H^\al$. Then $\al(\sigma_n(f) - f) \rightarrow 0$ as $n \to \infty$.
\end{lem}
We now obtain a similar decomposition for $H^\al$ in the context of the Blaschke factor $B(z)=z^n$, when $\al$ is a continuous rotationally symmetric norm. 

\begin{lem}\label{lem4.2}
Let $\al$ be a continuous rotationally symmetric norm. We can decompose the Banach space $H^\al$ as
    $$H^\al =  M_\al(z^n) \oplus z M_\al(z^n) \oplus\cdots\oplus z^{n-1} M_\al(z^n)$$
where $M_\al(z^n)$ = $\left [ H^\infty(z^n) \right ]_\al$ and $\oplus$ is the algebraic direct sum.
\end{lem}
\begin{proof}
Since $H^\infty(z^n) \subset H^\infty \subset H^\al$, so we have $M_\al(z^n) \subset H^\al$. Observe that for any natural number $k$ : $\al(z^k f)=\al(\lvert z^k f \rvert) = \al(f)$ for every $f \in H^\al$. Thus, multiplication by $z^k$ is an isometry on $H^\al$. This implies
\begin{center}
    $ M_\al(z^n) \oplus z M_\al(z^n) \oplus\cdots\oplus z^{n-1} M_\al(z^n) \subset H^\al$.
\end{center}

For the reverse inclusion, take an arbitrary $f \in H^\al$. We can write $f(z)$ as a Fourier series
   $$ f(z) = \sum_{j=0}^{\infty} \hat{f}(j) z^j $$
where $\hat{f}(j) = \int_\T f(z) z^{-j}\;dm$ for  $j = 0,1,2, \ldots$ \\
Let $\omega = e^{2 \pi i / n}$. For each $j = 1,2,\ldots,n$, define
$$ g_{n-j}(z)  = \frac{f(z)+ \omega^j f(\omega z) + \omega^{2j} f(\omega^2 z) +\cdots+ \omega^{(n-1)j} f(\omega^{n-1} z)}{n}. $$
Since $\al$ is a rotationally symmetric norm, so $f(\omega ^k z ) \in H^\al$. Thus, $g_{n-j} \in H^\al$.\\
For each $k = 0,1,\ldots,n-1$:
$$f(\omega^k z) = \sum_{j=0}^{\infty} \hat{f}(j) \; \omega^{kj} \; z^j.$$
\par
Next, a simple computation shows that the coefficient of $z^m$ in the expansion of $n \;g_{n-j}$ is given by
\begin{equation}\label{eq4.1}
\hat{f}(m) \; \left [ \frac{\omega^{n(j+m)} - 1}{\omega^{j+m} -1}  \right] \\
= \left\{ \begin{array}{lr}
0 \; \; , &  n \nmid {j+m} \\
\hat{f}(m) \; n \;  , & n \mid j+m \; .
\end{array} \right. 
\end{equation}
Equation (\ref{eq4.1}) is valid for all $m=0,1,2, \ldots$ and all $j=1,2,\ldots,n$. Therefore
\[ n\; g_{n-j} = n \hat{f}(n-j) z^{n-j} + n  \hat{f}(2n-j) z^{2n-j} + \cdots  .\]
Thus
\[ g_{n-j} = z^{n-j} h_{n-j}(z),   \]
where $h_{n-j}(z) = \hat{f}(n-j) + \hat{f}(2n-j) z^n + \cdots $.
Because $g_{n-j} \in H^\al$, we can see that $h_{n-j} \in H^\al$.
\par 
We now claim that $h_{n-j} \in M_\al(z^n)$. \\
Since $f(\omega^k z) \in H^\al$ for all $k =0,1,\ldots,n-1$ and $\sigma_l(f(\omega^k z))$ be the $l$-th Cesaro mean for $f(\omega^k z)$. By Lemma~\ref{lem4.1}, $$ \al(\sigma_l(f(\omega^k z)) - f(\omega^k z)) \longrightarrow 0 \; \; \;  \text{as} \; \;  l \to \infty.$$
A straightforward calculation shows that the sum 
\begin{equation}\label{eq4.2}
\sigma_l(f(z)) + \omega^j\sigma_l(f(\omega z)) + \omega^{2j}\sigma_l(f(\omega^2 z)) + \cdots + \omega^{(n-1)j}\sigma_l(f(\omega^{n-1} z)
\end{equation}
is of the form $z^{n-j}\kappa_l(z^n)$, where $\kappa_l(z^n)$ is a polynomial in $z^n$ which belong to $M_\al(z^n)$.\\
Further, the sum in equation (\ref{eq4.2}) converges in $H^\al$ to $f(z)+\omega^j f(\omega z) +\cdots+\omega^{(n-1)j} f(\omega^{n-1} z)$ as $l \to \infty$. 
Therefore, $z^{n-j} \kappa_l(z^n) \to n z^{n-j} h_{n-j}$ as $l \to \infty$ in $H^\al$.\\
That is, 
$$\frac{1}{n} \kappa_l(z^n) \longrightarrow h_{n-j}  \; \text{as}\; l \to \infty \; \text{in } \; H^\al.$$
Hence the claim follows. 
Therefore,
$f = h_0 + z h_1+\cdots+ z^{n-1} h_{n-1}$ belongs to 
$M_\al(z^n) \oplus z M_\al(z^n) \oplus\cdots\oplus z^{n-1} M_\al(z^n)$. This completes the proof.
\end{proof}
The proof of Lemma~\ref{lem4.2} relies on the fact that every function $f\in H^\al$ can be written as $f=f_0 + z f_1 +\cdots+z^{n-1}f_{n-1}$, and in turn we can express each $f_k$ in terms of the original $f$. This manipulation has been figured out only for $B(z)=z^n$. Hence, we make the following conjecture:
\begin{conj}
Let $\al$ be a continuous rotationally symmetric norm and $B$ be a finite Blaschke factor. Then is it true that $H^\al = e_{00} M_\al(B) \oplus e_{10} M_\al(B) \oplus\cdots\oplus e_{n-1,0}  M_\al(B)\;?$ 
\end{conj}

We are unaware whether the decomposition in Lemma~\ref{lem4.2} is valid for general $\|.\|_1$-dominating normalized gauge norms. So, we make the following conjecture:
\begin{conj}
Let $\alpha$ be a continuous $\|.\|_1$-dominating normalized gauge norm. Then is it true that $H^\al= M_\al(z^n) \oplus z M_\al(z^n)\oplus \cdots \oplus z^{n-1} M_\al(z^n) \; ?$
\end{conj}

As we have noticed that the class of rotationally symmetric norm is contained in the class of all $||.||_1-$dominating normalized gauge norm. The Theorem~\ref{theo3.5} and Theorem~\ref{theo3.6} are stands true in this case. Our main result of this section is to show that the decomposition which we have obtained for $H^\al$ is useful to sharpen the expression of all closed subspaces of $H^\al$ invariant under $H^\infty_1(z^n)$, but is not under $H^\infty(z^n)$. 

\begin{thm}\label{theo4.5}
Suppose $\al$ is a continuous rotationally symmetric norm and $\mathcal{M}$ be an $\al$-closed linear subspace of $H^\al$. Suppose $\mathcal{M}$ is invariant under $H^\infty_1(z^n)$, but not under $H^\infty(z^n)$. Then, 
\[  \mathcal{M} = \left (\sum_{i=1}^{k} \oplus \langle \varphi_i \rangle \right ) \oplus z^{2n} \; (J_1 M_\al(z^n) \oplus\cdots\oplus J_r M_\al(z^n))  \]
where $M_\al(z^n) = \left [H^\infty(z^n) \right ]_{\al}$ ,   $J_1,...,J_r$ ($r \leq n$) are $n$-inner functions and $k \leq 2r-1$. Furthermore, for all i = 1,2,\ldots,k , $\varphi_i$ = $(\beta_{1i} + \beta_{2i}z^n)J_1 + (\beta_{3i} + \beta_{4i}z^n)J_2  + \cdots  + (\beta_{2r-1,i} + \beta_{2r,i}z^n)J_r$.
\end{thm}

\begin{proof}
For $B(z) = z^n$, we have by Theorem~\ref{theo3.5} that there exist $n$-inner functions $J_1,\ldots, J_r$ such that $\mathcal{M}$ is of the form:
\begin{equation}\label{eq4.3}
\mathcal{M} = \left (\sum_{i=1}^{k} \oplus \langle \varphi_i \rangle \right ) \oplus z^{2n} \; \left [ J_1 M_\al(z^n) \oplus \cdots \oplus J_r M_\al(z^n) \right ]_{\al} 
\end{equation}
where the natural number $k$ and the functions $\varphi_i$ satisfy the hypothesis of Theorem~\ref{theo4.5}. 
Consider the closed subspace $W = \left [ J_1 M_\al(z^n) \oplus \cdots \oplus J_r M_\al(z^n) \right ]_{\al}$. 
\par
We first claim that 
\begin{equation}\label{eq4.4}
 W \cap H^\infty = J_1 H^\infty(z^n) \oplus\cdots\oplus J_r H^\infty(z^n).   
\end{equation} 
Note that $W = \left [ J_1 H^\infty(z^n) \oplus \cdots \oplus J_r H^\infty(z^n) \right ]_{\al}$. For any $g\in W\cap H^\infty$, there exists a sequence $\{g_m\} \subset J_1 H^\infty(z^n) \oplus\cdots\oplus J_r H^\infty(z^n)$ such that $g_m\to g$ in $H^\al$. We can write $g_m=J_1 g_m^{(1)} +\cdots+ J_r g_m^{(r)}$ where $g_m^{(i)}\in H^\infty (z^n)$.
\par
By Lemma~\ref{lem2.9}, we conclude that ${g_m^{(i)}}$ is a Cauchy sequence in $H^1$ and hence converges to some $g^{(i)}\in H^1$. Consequently, $g_m\to J_1 g^{(1)}+\cdots+J_r g^{(r)}$ in $H^1$.
Since $g\in H^\infty$, Lemma~\ref{lem2.9} implies $g^{(i)}\in H^\infty (z^n)$. So, $g \in J_1 H^\infty(z^n) \oplus\cdots\oplus J_r H^\infty(z^n)$. This means that $W \cap H^\infty \subset J_1 H^\infty(z^n) \oplus\cdots\oplus J_r H^\infty(z^n)$. The reverse containment is trivial.
\par
In order to finish off the proof, it remains to be established that $W= J_1 M_\al(z^n) \oplus \cdots \oplus J_r M_\al(z^n)$.\\
It is easy to see that $J_1 M_\al(z^n) \oplus \cdots \oplus J_r M_\al(z^n) \subset W$. We now establish the reverse containment. Let $f \in W$. As demonstrated in the proof of Theorem~\ref{theo3.5}, we can construct an outer function $O \in H^\infty(z^{2n})$ such that $ O f \in W \cap H^\infty$.\\ 
By equation (\ref{eq4.4}), there exist functions $g_1,\ldots,g_r\in H^\infty(z^n)\subset M_\al(z^n)$ such that:
\begin{equation}\label{eq4.5}
    O f = J_1 g_1 + J_2 g_2+\cdots+ J_r g_r.
\end{equation}
By Lemma~\ref{lem4.2}, we can express $f$ as
   \[ f =  f_1 + z f_2 +\cdots+ z^{n-1} f_n \]
for some $f_1,...,f_n \in M_\al(z^n)$. So, 
\begin{equation}\label{eq4.6}
    O f = O f_1 + z O f_2 +\cdots+ z^{n-1} O f_n.
\end{equation}
\par
For the tuple $(J_1,\ldots,J_r)$ of $n$-inner functions, we have a matrix $A = (\varphi_{ij})_{n \times r}$ of functions in $H^\infty (z^n)$ such that $A^* A = I$.\\
The $n$-inner functions $J_j$ can thus be written as
\begin{equation}\label{eq4.7}
J_j = \sum\limits_{i=0}^{n-1} z^i \varphi_{ij} \; \;  , \; \; 1 \leq j \leq r.
\end{equation}
From equations (\ref{eq4.5}), (\ref{eq4.6}), and (\ref{eq4.7}), we conclude that for each $i = 1,2,\ldots,n$:
\begin{center}
    $ O f_i$ = $\varphi_{i-1,1} g_1 + \varphi_{i-1,2} g_2 +\cdots+ \varphi_{i-1,r} g_r .$
\end{center}
Therefore,
\begin{equation}\label{eq4.8}
\begin{pmatrix}
    O f_1 \\
    \vdots\\
    Of_n
\end{pmatrix}
\; = \;  A \; \begin{pmatrix}
    g_1 \\
    \vdots\\
    g_r
\end{pmatrix}
.
\end{equation}
\smallskip
Taking the conjugate transpose, we have
\begin{equation}\label{eq4.9}
\begin{pmatrix}
    \overline{O f_1} & \cdots & \overline{O f_n}
\end{pmatrix}
\; = \; \begin{pmatrix}
    \overline{g_1} & \cdots & \overline{g_r}
\end{pmatrix} \; A^*.
\end{equation}
Multiplying equations (\ref{eq4.7}) and (\ref{eq4.8}) yields
$$\lvert O f_1\rvert^2 +\cdots+ \lvert O f_n\rvert^2 = \lvert g_1\rvert^2 +\cdots + \lvert g_r\rvert^2.$$
Therefore, for $1 \le i \le r$, 
\[ \left \lvert \frac{g_i}{O}  \right \rvert \leq  \lvert f_1\rvert +\cdots+ \lvert f_n\rvert. \]
Since $ \lvert f_{1} \rvert + \cdots + \lvert f_{n} \rvert \in L^\al $, we have, $\al \left( \dfrac{ g_i } {O} \right) = \al \left(\left \lvert \dfrac{ g_i } {O} \right \rvert \right)$ $\leq \al (\lvert f_1\rvert+\cdots+\lvert f_n\rvert) < \infty.$ So, $ \dfrac{g_{i} } {O} \in L^\al$.\\ 
Further, $ \dfrac{ g_i } {O} \in H^1$ because $O$ is outer. Hence, $ \dfrac{ g_i } {O} \in H^\al$. By Lemma~\ref{lem4.1}, the Cesaro means  $\sigma_l\left(\dfrac{g_i}{O}\right)$ converge to $\dfrac{g_i}{O}$ in $H^\al$. Since $\sigma_l\left(\dfrac{g_i}{O}\right)$ is a polynomial in $z^n$, so $\dfrac{g_i}{O}\in M_\al(z^n)$.\\
Hence, $f = J_1 \dfrac{g_1}{O} + J_2 \dfrac{g_2}{O}+\cdots+ J_r \dfrac{g_r}{O}$ 
$\in J_1 M_\al(z^n) \oplus\cdots\oplus J_r M_\al(z^n)$. This establishes that $ W = J_1 M_\al(z^n) \oplus\cdots\oplus J_r M_\al(z^n)$.\\ 
Hence equation (\ref{eq4.3}) can be written as:
\[ \mathcal{M} = \left (\sum_{i=1}^{k} \oplus \langle \varphi_i \rangle \right ) \oplus z^{2n} \;  (J_1 M_\al(z^n) \oplus\cdots\oplus J_r M_\al(z^n)). \]
\end{proof}

We now derive the $T_{z^n}$-invariant subspaces of $H^\al$, when $\al$ is a continuous rotationally symmetric norm. The proof proceeds on similar lines as that of Theorem~\ref{theo4.5}, so we only present a sketch. 

\begin{thm}\label{theo4.6}
Let $\al$ be a continuous rotationally symmetric norm and $\mathcal{M}$ be a closed subspace of $H^\al$. Suppose $\mathcal{M}$ is invariant under the operator $T_{z^n}$. Then there exist $n$-inner functions $J_1,\ldots,J_r$ ($r\le n$) such that
\[  \mathcal{M} = J_1 M_\al(z^n) \oplus\cdots\oplus J_r M_\al(z^n).  \]
\end{thm}

\begin{proof}
On the same lines as in the proof of Theorem~\ref{theo3.5}, we can show that $\mathcal{M} \cap H^\infty$ is a non-trivial weak*-closed subspace of $H^\infty$ which is invariant under $T_{z^n}$. \\
By Theorem~\ref{theo3.4}, there exist $n$-inner functions $J_1,\ldots,J_r$ (with $r \leq n$ ) such that
$$ \mathcal{M} \cap H^\infty = J_1 H^\infty(z^n) \oplus \cdots \oplus J_r H^\infty(z^n).$$
Therefore,
$$[\mathcal{M}\cap H^\infty]_\al = [J_1 H^\infty(z^n) \oplus \cdots \oplus J_r H^\infty(z^n)]_\al.$$
Further, in view of Lemma~\ref{lem3.11} it can be deduced that $[\mathcal{M} \cap H^\infty]_\al=\mathcal{M}$. Also, it has been established in Theorem~\ref{theo3.5} that $[J_1 H^\infty(z^n) \oplus \cdots \oplus J_r H^\infty(z^n)]_\al=[J_1 M_\al(z^n)\oplus\cdots\oplus J_r M_\al(z^n)]_\al$.\\
So, $\mathcal{M}$ takes the form:
$$\mathcal{M}= [J_1 M_\al(z^n) \oplus \cdots \oplus J_r M_\al(z^n)]_\al.$$
\par
In order to show that $\mathcal{M}$ takes the desired form it suffices to establish that 
$$\mathcal{M} \subset J_1 M_\al(z^n) \oplus \cdots \oplus J_r M_\al(z^n).$$
Note that for any $f\in M$, there exists an outer function $O$ such that $Of$ takes the form:
$$Of = J_1 g_1 +\cdots+ J_r g_r~~~(\; g_i\in H^\infty(z^n)\;).$$ 
Now, proceeding as in the proof of Theorem~\ref{theo4.5}, we find that
$f= J_1 \dfrac{g_1}{O} + J_2 \dfrac{g_2}{O}+\cdots+ J_r \dfrac{g_r}{O}$. This completes the proof because $\dfrac{g_i}{O} \in M_\al(z^n)$.
\end{proof}
\section{General Inner-Outer Factorization in $H^\al$}\label{sec5}

In this section, we derive the general inner-outer factorization of $H^\al$ functions. The definition of inner and $n-$inner functions have been discussed beforehand.   
The outer function in $H^\al$ is studied by Chen in \cite{Chen2014LebesgueAH}. Let $f\in H^\al$ is outer if $f$ is a cyclic vector for $H^\infty$ acting on $H^\al$ which means $[f \cdot H^\infty]_\al = H^\al$.
We now define an $n$-outer function in $H^\al$.
\begin{defn}\label{defn5.1}
Suppose $\al$ is continuous rotationally symmetric norm and $f\in H^\al$. We say $f$ is $n$-outer function in $H^\al$ if:
\begin{equation}\label{eq5.1}
[f \cdot H^\infty(z^n)]_\al = \; p(z)\; M_\al(z^n)
\end{equation}
where $p(z)$ is a polynomial of degree less than $n$ and $M_\al(z^n) = [H^\infty(z^n)]_\al$. 
\end{defn}
\begin{rem}\label{rem5.2}
Observe that for any $f \in H^\al$, $[f \cdot H^\infty(z^n)]_\al = \bigvee\limits_{k=0}^{\infty} \{ f \cdot z^{nk}\}$ where $\bigvee\limits_{k=0}^{\infty} \{ f \cdot z^{nk}\}$ is an $\al$-closed linear span of $\{  f \cdot z^{nk} : k = 0,1,\ldots \}$.
\end{rem}
\noindent
It is trivial to show that $\bigvee\limits_{k=0}^{\infty} \{ f \cdot z^{nk}\} \subset [f \cdot H^\infty(z^n)]_\al$.\\
For the reverse inclusion, let $g \in H^\infty(z^n)$. 
Then, the sequence of Cesaro means $\{ \sigma_m(g) \}$ converges to $g$ in the weak* topology. Therefore,
$$\int_{\T} \sigma_m(g) h dm \longrightarrow  \int_{\T} g h dm ~~~~\forall ~~h \in L^1.$$
Thus, for any non-zero element $u \in \mathcal{L}^{\al^\prime}$, we see that
$$\int_{\T} \sigma_m(g) f u dm \longrightarrow  \int_{\T} g f u dm.$$
Therefore, sequence $\{\sigma_m(g) f\}$ converge weakly to $gf$ in $L^\al$. But each $\sigma_m(g)$ is a polynomial in $z^n$, so $\sigma_m(g) f \in \bigvee\limits_{k=0}^{\infty} \{ f \cdot z^{nk}\}$. Since $\bigvee\limits_{k=0}^{\infty} \{ f \cdot z^{nk}\}$, being a closed subspace of $H^\al$ is also weakly closed, and hence $g \cdot f \in \bigvee\limits_{k=0}^{\infty} \{ f \cdot z^{nk}\}$. 
\begin{rem}
Note that when $f$ is $n$-outer function in $H^\al$ such that $f \in M_\al(z^n)$ then 
$$ [f \cdot H^\infty(z^n)]_\al = M_\al(z^n).$$
\end{rem}
\noindent
This can be seen as follows:\\
Since $f$ is $n$-outer in $H^\al$, so we can write $f(z)=p(z) g(z)$ for some $g\in M_\al(z^n)$. Since the $deg\; p(z)<n$, we can write $p(z) = a_0 + a_1 z +\cdots+ a_{n-1} z^{n-1}$. Therefore,
$$f(z) = a_0 g(z) + a_1 z g(z) +\cdots+ a_{n-1} z^{n-1} g(z).$$
But $f \in M_\al(z^n)$. This is possible only when $a_1, a_2,\ldots, a_{n-1}$ to be 0. Therefore $p(z)$ is a constant polynomial and the result follows.
\par
\smallskip
We now derive a factorization of $H^\al$ functions into the sum of products of $n$-inner and $n$-outer functions. We use Theorem~\ref{theo4.5} to obtain a general inner-outer factorization of $H^\al$ functions.

\begin{thm}\label{theo5.4}
Suppose $\al$ is continuous rotationally symmetric norm and $f \in H^\al$ then
$$ f = J_1 f_1 + \cdots + J_r f_r \; \; (r \leq n)$$
where, each $J_i$ is $n$-inner and $f_i \in M_\al(z^n)$ is $n$-outer for all $i = 1,\ldots,r$. 
\end{thm}

\begin{proof}
Consider $\mathcal{M} = [f \cdot H^\infty(z^n)]_\al$. Then $\mathcal{M}$ is a closed subspace of $H^\al$ which is invariant under $T_{z^n}$. By Theorem~\ref{theo4.5}, there exist $n$-inner functions $J_1,\ldots,J_r$ such that
$$ \mathcal{M} = J_1 M_\al(z^n) \oplus \cdots \oplus J_r M_\al(z^n).$$
Since $f \in \mathcal{M}$ , there exists $f_1 , \ldots, f_r \in M_\al(z^n)$ such that
$$ f = J_1 f_1 + \cdots + J_r f_r.$$
We claim that each $f_i$ is $n$-outer.\\
If possible assume that some $f_i$ is not $n$-outer function, then
$$ [f \cdot H^\infty(z^n)]_\al  \subsetneq M_\al(z^n).$$
We can write
$$ J_i [f \cdot H^\infty(z^n)]_\al \subsetneq J_i M_\al(z^n).$$
Note that, 
\begin{align*}
\mathcal{M} &= [f \cdot H^\infty(z^n)]_\al\\
&\subset [ J_1 f_1 H^\infty(z^n) \oplus \cdots \oplus J_i f_i H^\infty(z^n) \oplus \cdots \oplus J_r f_r H^\infty(z^n)]_\al\\ 
&\subset [ J_1 [f_1 \cdot H^\infty(z^n)]_\al \oplus \cdots \oplus J_i [f_i\cdot H^\infty(z^n)]_\al \oplus \cdots \oplus J_r [f_r \cdot H^\infty(z^n)]_\al]_\al\\
&\subsetneq [ J_1 M_\al(z^n) \oplus \cdots \oplus J_i M_\al(z^n) \oplus \cdots \oplus J_r  M_\al(z^n)]_\al\\
&= \mathcal{[M]_\al}
\end{align*}
i.e., $\mathcal{M} \subsetneq \mathcal{M}$ which is a contradiction. Therefore, each $f_i$ is $n$-outer function and thus the proof is complete.
\end{proof}

\begin{rem}
The factorization established in Theorem~\ref{theo5.4} is different to the one presented in \cite{Lance1997multiplication} for classical $H^p$ spaces. In \cite{Lance1997multiplication}, the authors decomposed every $H^p$ function as a product of $z^n$-$p$-inner and $p$-outer functions. Whereas, we are decomposing $H^\al$ in terms of $n$-inner and $n$-outer functions. \\
A function $\varphi \in H^\infty$ is called a $z^n$-$p$-inner function if $\int_\T  |\varphi|^p z^n dm<\infty$. We have not attempted to define a concept on the similar lines for $H^\al$.
\end{rem}

\begin{cor}[Singh and Thukral~\cite{singh1997multiplication}]
Suppose $f$ is in $H^2$ then
$$f = \varphi_1 f_1 + \cdots + \varphi_r f_r \; \; (r \leq n)$$ 
where for each $i$, $f_i \in H^2(z^2)$ and is $n$-outer and $\varphi_i$ is $n$-inner.
\end{cor}

\begin{thm}
Every $n$-outer function which belongs to $M_\al(z^n)$ is an outer function.
\end{thm}
\begin{proof}
Let $f \in M_\al(z^n)$ be an $n$-outer function. By Remark~\ref{rem5.2}, we have
\begin{equation}
[f \cdot H^\infty(z^n)]_\al = M_\al(z^n).
\end{equation}
In order to prove that $f$ is an outer function, we need to show that $H^\al = [f \cdot H^\infty]_\al$.\\
\smallskip
It is trivial to note that $[f \cdot H^\infty]_\al \subset H^\al$. So only the reverse inequality needs to be established.\\
Let $h \in H^\al$ and write $h$ as 
$$ h = h_1 + z h_2 +\cdots+z^{n-1}h_n$$
for some $h_1,\ldots,h_n \in M_\al(z^n)$.\\
By equation (\ref{eq5.1}), we see that each $h_i \in [f \cdot H^\infty(z^n)]_\al $, hence there exists a sequence $
\{fh_k^{(i)}\} \subset f \cdot H^\infty(z^n)$ such that
$$\al(fh_k^{(i)} - h_i) \to 0 \; \; \text{as} \; \; k \to \infty.$$
Therefore,
$$\al(z^i f h_k^{(i)} - z^i h_i) \to 0  \; \;  \text{as} \; \; k \to \infty.$$
This means that
$$ f \cdot \left (h_k^{(1)} + z h_k^{(2)} + \cdots + z^{n-1} h_k^{(n)}\right) \longrightarrow h  \; \; \text{in} \; \; H^\al \; \; \text{as} \; \; k \to \infty.$$
Thus, we have $h \in [f \cdot H^\infty]_\al$. So the result follows.
\end{proof} 

\subsection*{Acknowledgments}
The authors would like to thank Prof. Sachi Srivastava, University of Delhi, India for useful comments during the course of the present research.

\nocite{Hoffman1962banach}
\bibliographystyle{spmpsci}
\bibliography{invariance_in_Halpha.bib}

\end{document}